\documentclass[12pt]{amsart}
\usepackage{amsmath,amsfonts,amssymb,amscd,amsthm,epsf}
\usepackage{graphics}
\usepackage[latin1]{inputenc}

\usepackage[frenchb]{babel}
\newcommand{\bbC}{{\mathbb C}}
\newcommand{\bbZ}{{\mathbb Z}}
\newcommand{\bbR}{{\mathbb R}}

\newtheorem{theorem}{Théorème}[section]
\newtheorem{proposition}[theorem]{Proposition}
\newtheorem{lemma}[theorem]{Lemme}
\newtheorem{corollary}[theorem]{Corollaire}
\newtheorem{remark}{Remarque}[section]

\theoremstyle{definition}
\newtheorem{definition}[theorem]{D\'efinition}

\theoremstyle{remark}

\newcommand{\be}{\begin{equation}}
\newcommand{\ee}{\end{equation}}
\newcommand{\ben}{\begin{equation}\nonumber}
\newcommand{\mq}{mécanique quantique }

\begin{document}

\title[Échelles de temps]{ Échelles de temps pour l'évolution quantique à petite constante de Planck}
\author[Thierry PAUL]{T. Paul\\ CNRS et D.M.A., Ecole Normale Supérieure}
\address{DMA, École Normale Supérieure, 45 rue d'Ulm 75230 Paris Cedex 05}
\email{paul@dma.ens.fr\\ \  http://www.dma.ens.fr/$\sim$paul}
\thanks{Expos\'e au s\'eminaire X-EDP 27-12-2007}
\setcounter{page}{0}
\date{}

\maketitle
\tableofcontents

\section{Introduction}
La limite semiclassique pour l'équation de Schrödinger est née dès le début de la mécanique quantique. Elle l'a même précédé si l'on considère que les
théories de Heisenberg et Schrödinger ont précisément été inventées pour donner un cadre dynamique aux conditions de Bohr-Sommerfeld. Appelée ``principe de
correspondance" en physique et attachée au théorème d'Egorov en analyse microlocale, l'approximation semiclassique assure donc la transition quantique-classique, 
pour ce qui
concerne la partie ``équation de Schrödinger" de l'évolution quantique. Il faut tout de suite remarquer que la partie ``mesure quantique" de l'évolution, indéterministe et intrinsèquement
aléatoire, n'a pas de ``correspondant" classique. Nous discuterons brièvement à la fin de cet article un lien possible entre ``mesure" et limite semiclassique,
précisément pour la limite à temps long.

La \mq  contient dans son formalisme une constante dimensionnée, la constante de Planck $\hbar$, qui  a la dimension d'une action (position $\times$ impulsion) et qui assure ainsi
l'homogénéité nécessaire à la transformée de Fourier. C'est aussi $\hbar$ qui contrôle  la taille des oscillations inhérentes à la théorie car il ne faut pas oublier que la limite
semiclassique se réalise par oscillations et, de ce point de vue, est beaucoup plus subtile que la limite non-relativiste.

Dans ce texte nous présentons quelques résultats récents concernant l'approximation semiclassique à temps long, c'est-à-dire que nous allons considérer l'équation de
Schrödinger:
\be
i\hbar\partial_t\psi=(-\hbar^2\Delta+V(x))\psi
\ee
pour des temps $t\leq T(\hbar)$ avec diverses fonctions $T(\hbar)\to+\infty$ lorsque $\hbar\to 0$. Les formes explicites des fonctions $T(\hbar)$ vont dépendre des 
propriétés
dynamiques des flots classiques correspondants, mais nous allons, dans les cas considérés, exhiber trois échelles de temps qui donneront une limite différente.
\begin{itemize}
\item $T_1(\hbar)$ pour laquelle la limite semiclassique est exactement la mécanique classique

\item $T_2(\hbar)$ pour laquelle la limite semiclassique ne sera plus classique, mais pourra néanmoins se comprendre à partir des propriétés dynamiques de celle-ci, en
particulier celles de stabilité (c'est l'échelle de temps pour laquelle les paquets d'ondes se délocalisent).
\item $T_3(\hbar)$ pour laquelle des phénomènes typiquement quantiques vont perdurer à la limite classique (et permettent des reconstructions et phénomènes d'ubiquité).
\end{itemize}
La physique récente (celle des atomes froids) nous apprend que la correspondance quantique/classique $\Leftrightarrow$ micro/macroscopique n'est plus systématiquement d'actualité.
On arrive maintenant à garder excités des atomes dont la taille est de l'ordre de celle d'une bactérie. De plus les échelles de temps pendant lesquelles on est capable de maintenir
des systèmes quantiques simples dans de tels états  est de l'ordre de la seconde \cite{HR}. De plus la reconstruction de paquets d'ondes observée expérimentalement dans des
supraconducteurs constitue une branche très active de la physique moderne \cite{bl}, en particulier en vue d'applications a l'information et le calcul quantique \cite{YS}.

Pour corroborer ces remarques signalons que la limite semiclassique ne correspond pas toujours à faire arbitrairement varier la ``constante" de Planck. Donnons deux exemples :
\begin{itemize}

\item le cas des systèmes atomiques, avec des  ``scalings" qui redonne la limite semiclassique. L'exemple le plus simple est bien sûr l'atome d'hydrogène :
\[
-\hbar²\Delta-\frac 1 {|x|}
\]
unitairement équivalent (par une dilatation par $\hbar^{-\frac  2 3}$) à 
\[
\hbar^{\frac 2 3}\left(-\Delta-\frac 1 {|x|}\right),
\]
et donc pour lequel la limite à hauts états excités est bien la limite semiclassique.

\item le cas des spins. On peut montrer en effet que l'espace des états de $n$ spins $\frac 1 2 $ en interaction invariante par permutation peut être réalisé par celui d'un seul
spin $n$. la limite $n\to\infty$ correspond alors à la limite à grand  spin d'une seule particule, c'est-à-dire la limite semiclassique.\end{itemize}
 \vskip 0.5cm
 \textit{La limite semiclassique est donc bien
contenue dans la mécanique quantique, elle en constitue un ``bord"}.
\vskip 0.5cm

L'approximation semiclassique à temps long a été étudiée dans
\cite{CR} pour les états cohérents, et dans \cite{BGP} puis \cite{BOR} en ce qui 
concerne le théorème d'Egorov. Les éléments de matrice diagonaux entre états cohérents, dans le cas monodimensionnel stable pour le propagateur et instable pour la
propagation d'observables ont été considérés dans \cite{DR} et \cite{RB}
respectivement.

\section{Le problème} 
Nous allons considérer l'équation de Schrödinger dépendant du temps
\be
i\partial_t\psi=H\psi
\ee
sur $L^2(\bbR^n,dx)$ avec une condition initiale sous la forme d'un état cohérent. Ici $H$ sera un opérateur pseudo-différentiel à petit paramètre (semiclassique) de symbole de Weyl $h(x,\xi)$,

Les états cohérents, inventés par Schrödinger lui-même dans l'un des articles originaux de 1926, \cite{ES}, constituent un ensemble surcomplet de vecteurs de $L^2(\bbR^n)$. Nous
allons considérer ici des états cohérents de ``forme" quelconque (pas seulement gaussienne), pour plusieurs raisons. Tout d'abord les résultats présentés sont valables dans ce cadre
général. Ensuite parce que nous verrons que la forme gaussienne, qui perdure pour des temps ``raisonnables", n'est plus suffisante pour les échelles de temps que nous allons
considérer.
\begin{definition}\
soit $(q,p)\in\bbR^{2n}=T^*\bbR^n$ et $a\in\mathcal S(\bbR^n)$ (la classe de Schwartz sur $\bbR^n$) avec $||a||_{L²}=1$. On définit
\be\label{ec}
\psi^a_{qp}(x)=\hbar^{-n/4}a\left(\frac{x-q}{\sqrt\hbar}\right)e^{i\frac{px}\hbar}.
\ee
on dit que $\psi^a_{qp}$ est un état cohérent ``en $(q,p)$" et de ``symbole $a$".
\end{definition}

Rappelons quelques propriétés des états cohérents :
\vskip 0.5cm
\begin{itemize}
\item complétion : $\int_{\bbR^{2n}}|\psi^a_{qp}><\psi^a_{qp}|\frac{dqdp}{\hbar^{2n}}=\mbox{Id}_{L^2(\bbR^{2n})}$ au sens faible
\item mesure de Wigner : \textit{la $\hbar$-transformée de Wigner de $\psi^a_{qp}$ en $(x,\xi)$ est égale à la ($\hbar=1$)-transformée de Wigner  de $a$ en $(\frac{x-q}{\sqrt\hbar},
\frac{\xi-p}{\sqrt\hbar})$}
\item quantification de Töplitz : \textit{(presque) tout opérateur pseudo-différentiel s´écrit (modulo $\hbar^\infty$)sous la forme 
$\int_{\bbR^{2n}}h(q,p)|\psi^a_{qp}><\psi^a_{qp}|\frac{dqdp}{\hbar^{2n}}$}. 
\end{itemize}
\vskip 0.5cm
Un des avantages à travailler avec des états cohérents est que l'on reste toujours local. 
Les hypothèses suffisantes sur le hamiltonien quantique  $H$ sont donc qu'il soit essentiellement auto-adjoint, et de domaine contenant les états cohérents. Par exemple $H$
donné par la quantification de Weyl de $h$, c'est à dire de noyau intégral :
\[
H(x,y)=\int h\left(\frac{x+y}2,\xi\right)e^{i\frac{\xi(x-y)}\hbar}\frac {d\xi}{\hbar^n}
\]
et $h(x,\xi)=\xi^2+V(x)$, $V$ polynôme croissant à l'infini, ou bien $h\  C^\infty$ uniformément borné sur $\bbR^{2n}$ ainsi que toutes ses dérivées \cite{MA}.

\section{Résultats généraux}
Commençons par rappeler le lemme fondamental :
\begin{lemma}
Soit $H$ auto-adjoint et soit $\psi^t$ et $\phi^t$ tels que :
\[
i\hbar\partial_t\psi^t=H\psi^t+R^t,
\]
\[
i\hbar\partial_t\phi^t=H\phi^t,
\]
\[
\psi^0-\phi^0=\alpha(0).
\]
Alors, pour $t\geq 0$, 
\[
||\psi^t-\phi^t||=O(||\alpha(0)||)+O\left(\frac{\int_0^t||R^s||ds}\hbar\right)=O\left(||\alpha(0)||+t\frac{sup_{0\leq s\leq t}||R^s||}\hbar\right).
\]
en particulier, si $\alpha_0=0$,
\[
||\psi^t-\phi^t||=\left(t\frac{sup_{0\leq s\leq t}||R^s||}\hbar\right).
\]
\end{lemma}
\begin{proof}
 On obtient facilement des équations précédentes :
\[i\hbar\partial_t(\psi^t-\phi^t)=H(\psi^t-\phi^t)+R^t.
\]
Dénotant $\psi^t-\phi^t:=e^{-itH/\hbar}\alpha(t)$ on a :
\[
\alpha(t)=\alpha(0)+\frac 1 {i\hbar}\int_0^t e^{isH/\hbar}R^sds.
\]
\end{proof}

L'importance des états cohérents en analyse semiclassique provient du résultat suivant, énoncé par Schrödinger dans le cas  de l'oscillateur harmonique. 

Soit $\Phi^t$ le flot hamiltonien associé
au symbole principal $h$ du hamiltonien quantique $H$. A $\Phi^t$ on peut associer sa dérivée $d\Phi^t_{x\xi}$ au point $(x,\xi)$, c'est-à-dire :
\[
d\Phi^t_{x\xi}:\ \ \ T_{(x,\xi)}(T^*\mathcal M)\to T_{\Phi^t(x,\xi)}(T^*\mathcal M)
\]
$d\Phi^t_{x\xi}$ est, pour tout temps, une matrice symplectique et donc, par la représentation métaplectique, on peut lui associé un opérateur unitaire ds $L^2$, $M^t$, d'ailleurs défini
par l'équation :
\[
i\dot M^t=H^2(\Phi^t(x,\xi))M^t
\]
où $H^2(\Phi^t(x,\xi))$ est le (Weyl) quantifié du hessien de $H$ en $\Phi^t(x,\xi)$.

Nous allons supposer que, pour tout $t\geq 0$ :
\[
||d\Phi^t_{x\xi}||\leq e^{\mu(x,\xi)t}
\]
où $||.||$ est la norme opérateur et :
\vskip 0.5cm
\centerline{$\mu(.)\geq 0$ est une fonction $\alpha$-höldérienne pour un $\alpha>0$.}
\vskip 0.5cm
Un candidat naturel pour $\mu(x,\xi)$ est le supremum, sur la surface d'énergie de $(q,p)$, du hessien du hamiltonien. Mais l'intérêt du résultat qui va suivre est de
considérer les plus petits $\mu(x,\xi)$ possibles.

Soit finalement 
\[
l(t)=\int_0^t(\xi \dot x-h)ds\ \ \ \mbox{l'action lagrangienne du flot}
\]

\begin{theorem}\cite{TP3}
il existe une constante $C$, ne dépendant que du symbole $h$ autour de $h(x,\xi)=h(q,p)$ telle que, pour tout $t\geq 0$,
\be
||e^{-i\frac{tH}\hbar}\psi_{x\xi}^a-e^{i\frac{l(t)}\hbar}\psi_{\Phi^t(x,\xi)}^{M(t)a}||_{L^2}\leq Cte^{3\mu(x,\xi)t}\hbar^{\frac 1 2 }
\ee
\end{theorem}

On voit donc apparaître une première échelle de temps, celle pour laquelle la propagation quantique ``suit" la mécanique classique. Cette échelle est donnée par la condition que
\[
te^{3\mu(x,\xi)t}<C\hbar^{-\frac 1 2},
\]
condition remplie pour $t\leq T_1(\hbar)$ avec, pour $\epsilon>0$, \[
T_1(\hbar)=(\frac 1 6-\epsilon)\log{\hbar^{-1}}.
\]
\begin{remark} dans le cas où le flot est stable, on a que $||d\Phi^t_{x\xi}||\leq C\times t$ (il suffit de penser au flot libre de hamiltonien $\xi²$ pour s'en convaincre, en
notant que, dans ce cas, $d\Phi^t_{x\xi}=\left(\begin{array}{cc}1&t\\0&1\end{array}\right)$). On se convainc alors aisément que l'échelle $T_1$ est donnée par 
\[
T_1(\hbar)= C\hbar^{-\frac 1 6}.
\]
\end{remark}
Nous allons dans le reste de cet article considérer plusieurs cas particuliers où l'on pourra aller beaucoup plus loin dans les grands temps. mais commençons par un cas
d'école.

\section{Un cas très simple}
Considérons une particule libre sur le cercle. c'est-à-dire le cas où l'équation de Scrôdinger se réduit à 
\be
i\hbar\frac{\partial}{\partial t}\psi^t=-\hbar^2\Delta\psi^t\ \ \ 
\mbox{sur } L^2(S^1).
\ee 
Une simple décomposition en série de Fourier de la condition initiale $\psi^0(x)=\sum c_j e^{ijx}$ montre facilement que le flot quantique est périodique avec 
période $T=\frac{2\pi}\hbar$. En effet on a immédiatement que $\psi^t(x)=\sum c_j e^{-it\hbar j^2}e^{ijx}$. Cette
période est typiquement quantique et n'a pas d'équivalent classique, le flot classique n'ayant pas de période globale.

Précisons un peu. Puisque le spectre de l'opérateur $-\hbar\frac{d²}{dx²}$ sur le cercle est $\{\hbar²n²,\ n\in\bbZ\}$, si l'on décompose la condition initiale sur la base de
Fourier :
\be
\psi^0(x)=\sum c_ne^{inx}
\ee
on obtient immédiatement :
\be
\psi^t(x)=\sum c_ne^{-it\hbar n²}e^{inx}
\ee
et donc $\psi^{k\frac{2\pi}\hbar}=\psi^0,\ \forall k\in\bbZ$. Le flot quantique est donc globalement périodique, de période $T_Q=\frac{2\pi}\hbar$, alors que le flot classique
\be
\left\{\begin{array}{ccl}
\dot\xi&=&0\\
\dot x&=&2\xi\end{array}\right.
\ee

 est
périodique (puisque sur le cercle) mais de période variant avec la condition initiale.
\begin{remark} cependant on peut remarquer que, si l'on réduit l'espace de phases aux impulsions \textit{quantifiées}, c'est-à-dire de la forme $\xi=n\hbar,\ n\in\bbZ$, alors le
flot classique restreint est globalement périodique. Mais la période (minimale) est égale à $T_C=\frac\pi\hbar$, c'est-à-dire moitié de la période quantique. Il est amusant de
remarquer que le même phénomène se produit pour l'oscillateur harmonique (à cause de l'indice de Maslov) :
\ben
T_Q=2T_C
\ee
\end{remark}
Prenons maintenant le cas d'un état cohérent (gaussien) centré à l'origine.  Un
état cohérent (gaussien) sur le tore est simplement la périodisation d'un état cohérent habituel, ce qui donne, par la formule de Poisson et dans la cas $(q,p)=(0,0)$ :
\be\label{dsa}
\psi_0(x)=\left(\frac\pi\hbar\right)^{1/4}\sum e^{-\frac{n²}2\hbar}e^{inx}.
\ee
Le résultat précédent s'applique bien sûr. Mais considérons  l'évolution quantique pour un temps multiple rationnel de $T_Q$ de la forme :
\be
t_{\frac qp}:=\frac pq\frac{2\pi}\hbar.
\ee
Quelques manipulations de congruence donnent :
\[\begin{array}{ccl}
\psi^{t_{\frac qp}}(x)&=&\left(\frac\pi\hbar\right)^{1/4}\sum e^{i2\pi\frac pq n²}e^{-frac{n²}2\hbar}e^{inx}\\
&=&\left(\frac\pi\hbar\right)^{1/4}\sum_{l\in\bbZ,k=1\dots q}e^{i2\pi\frac pq (lq+k)²}e^{-\frac{(lq+k)²}2\hbar}e^{i(lq+k)x}\\
&=&\left(\frac\pi\hbar\right)^{1/4}\sum_{k=1\dots q}e^{i2\pi\frac p q k²}]\frac 1 q \sum_{j=1}^q\psi_{j\frac{2\pi}q}\end{array}
\]
où $\psi_{j\frac{2\pi}q}$ est un état cohérent en $(q=j\frac{2\pi}q,\xi=0)$.

On obtient donc un phénomène cette fois exclusivement quantique, sans correspondant classique : une reconstruction du paquet d'onde en plusieurs cites. Une ubiquité classique, trace
de la persistance à la limite classique d'effets quantiques.

Remarquons aussi que, si l'on considère (formellement) des temps $t_{\frac qp}$ avec $q\to\infty$, les points de reconstruction s'accumulent sur le cercle. En particulier, lorsque
$q\sim\hbar^{-\frac 1 2 }$, la distance entre les points est du même ordre que la
largeur des états correspondants. On n'a donc plus de localisation, et l'état du système est
totalement dispersé, ce que l'on peut montrer rigoureusement (voir aussi \cite{GO} pour une étude non-semiclassique).

La question naturelle est maintenant de se demander si les résultats précé-\\dents perdurent pour des hamiltoniens plus généraux. En particulier pour l'extension la plus immédiate qui
consiste à rajouter au symbole \textit{libre} $\xi²$ des termes d'ordre supérieur, à commencer par $\xi^3$. Si l'on reprend la formule (\ref{dsa}), on s'aperçoit que les nombres
$n$ signifiants sont ceux de l'ordre de $|n|\leq \hbar^{-\frac 1 2 }$. Si l'on considère, par exemple, un hamiltonien contenant des termes cubiques dans les impulsions, on aura :
\be
\psi_0^{\frac{2\pi}\hbar} (x)=\left(\frac\pi\hbar\right)^{1/4}\sum e^{i2\pi(n^2+\hbar n^3)}e^{-\frac{n²}2\hbar}e^{inx}=\sum e^{i2\pi n²(1+\hbar n)}e^{-\frac{n²}2\hbar}e^{inx}.
\ee
Pour $|n|\leq \hbar^{-\frac 1 2 }$, on a $|\hbar n|\leq \hbar ^{\frac 1 2 }\Rightarrow 1+\hbar n\sim 1$. On pourrait donc penser que les termes cubiques (et d'ordres
supérieurs) ne troublent pas la reconstruction. Nous allons voir qu'il n'en est rien, car la reconstruction d'un paquet d'onde est une affaire de cohérence de phases, et donc
c'est la distance des $n^2, n^3$ etc.... au réseau périodique qui importe et non leur taille.

\section{Le cas stable}\label{sta}
\subsection{Cas général}
Considérons  le cas d'un hamiltonien de symbole  $h C^\infty,\ h=h(\xi)=\frac{\xi^2}2+c\xi^3+d\xi^4+0(\xi^5)$, c'est-à-dire $ H:=h(-i\hbar\partial_x)$, toujours sur le cercle.
Et prenons toujours pour condition initiale un état cohérent à l'origine, donné par (\ref{dsa}) :

 On obtient  immédiatement que
 \be
 \psi^t(x):=e^{-i\frac{tH}\hbar}\psi_0(x)=\left(\frac\pi\hbar\right)^{1/4}\sum_{n\in\bbZ}e^{-\frac{m^2\hbar}2}e^{-i\frac{tH(m\hbar)}\hbar}e^{imx}.
 \end{equation}

Le résultat est le suivant :

\begin{proposition}\cite{TP1}.
 Soit $t=\frac{s 2\pi}\hbar, s\in\mathbb N$ fixé non nul. Alors pour $0<x\leq 2\pi$,
 \ben
 \psi^t(x)\ =\ \frac{\pi^{1/4}} {(sc)^{1/4}}
\sum_{k\geq 0}\frac {e^{-\frac{x +k2\pi}{sc}+i\frac d s(\frac{x +k2\pi}{c})^2}}{(x+k2\pi+\pi/4)^{\frac 1 4}}\sin(\frac 2 3 (\frac{x+k2\pi+\pi/4}{sc})^{3/2}) + 0(\hbar ^{\frac 1 2}).
 \end{equation}
\end{proposition}
On a en fait un développement asymptotique complet en puissance de $\hbar^{\frac 1 2 }$. De plus pour $x\sim 0$, $x=0(\hbar^{\frac 1 3})$
 \ben
 \psi^t(x)\ =\ \frac{\hbar^{-1/12}}{(sc)^{1/3}}Air(\frac x{(sc)^{1/3}}) +0(\hbar^{1/6})
 \end{equation}
Signalons aussi que l'on peut calculer de même pour des conditions initiales 
états cohérents de symbole $a$ quelconque. Enfin on peut remarquer que lorsque $s$ augmente la localisation exponentielle diminue, les effets dispersifs redeviennent actifs.

Ce qu'il faut retenir de cette formule un peu compliquée est le fait que la localisation a disparue, au moins dans le sens semiclassique : l´état ne se localise pas mieux lorsque
$\hbar\to 0$. Il y a deux manières de ``retrouver"la localisation \begin{itemize}
\item lorsque la constante $c$ tend vers $0$ : c'est la limite où la partie cubique du hamiltonien disparaît et l'on retrouve le résultat précédent, modulé par le terme quartique.
\item lorsque l'on part d´états ``comprimés" (en Fourier). C'est ce que nous allons voir maintenant.
\end{itemize}

\subsection{La reconstruction des états comprimés}
Considérons maintenant pour condition initiale un état comprimé de la forme, pour $\epsilon>0$ :

\begin{equation}\label{4}
 \psi^\epsilon (x):= \left(\frac\pi{\hbar^{1-\epsilon}}\right)^{1/4}\sum_{n\in\bbZ}e^{-\frac{n^2\hbar^{1-\epsilon}}2}e^{inx}.
 \end{equation}

Alors on a :
\begin{theorem}\cite{TP1}.
Soit $t=\frac{s 2\pi}{\hbar}, s\in\mathbb N$ fixé non nul. Alors pour $0<x<2\pi$,
  \begin{equation}\label{100}
 e^{i\frac{t H}\hbar}\psi(x)\ =\ \frac{1} {\hbar^{\epsilon/2}(sc/4)^{1/4}}
e^{-\frac{x }{sc\hbar^\epsilon}}+O(\hbar^{\frac 1 2})
 \end{equation}
 
 \end{theorem}
(autour de $0$ l'asymptotisme ne change pas par rapport au cas non-comprimé).

Ce résultat est apparemment surprenant puisqu'il indique qu'il faut partir d'une condition initiale \textit{moins} bien localisée pour obtenir une reconstruction. La raison
en est que, étant moins bien localisée en position, la condition initiale l'est mieux en Fourier, et donc est moins sensible au effets non-linéaires créés par les termes
d'ordre au moins cubique du hamiltonien. C'est
là une particularité de l'approximation semiclassique que de jouer avec les aspects quantiques (inégalités de Heisenberg) et classiques (flots non-linéaires), afin de produire des
effets contrintuitifs.

\subsection{Le cas de la trajectoire périodique stable}
Montrons brièvement comment les résultats précédent se généralisent au cas multidimensionnel d'une trajectoire périodique classique linéairement stable
non-dégénérée. 

La théorie des formes normales (quantiques), \cite{isz}, \cite{GP}, nous apprend que, près d'une trajectoire classique elliptique  $\gamma$ d'un hamiltonien sur
$T^*\mathcal M$, où $\mathcal M$ est une variété à $n $ dimensions, et sous une condition de non résonance, on peut construire un opérateur Fourier integral qui, microlocalement autour de $\gamma$, entrelace le hamiltonien quantique original sur
$L^2(\mathcal M)$
à un hamiltonien sur $L^2(\mathbb S^1\times \mathbb
 R^{n-1})$ de la forme
 \ben
 H'=H'(-i\partial_x,h_1,\dots,h_{n-1})
 \end{equation}
 où les  $h_j$ sont des oscillateurs  harmoniques dans les variables $y_j$. Présentons le résultat dans le cas de la dimension $2$.
 Écrivons :
\ben
 H'( \tau,h_1)=\tau+h_1+a\tau h_1   +b\tau^2+c\tau^3+d\tau^4+H"
 \ee 
 On peut clairement ne pas considérer les termes $\tau$ et $h_1$ (changement de repère, voir \cite{TP1}) et prendre :
 \ben
 H'( \tau,h_1)=a\tau h_1   +b\tau^2+c\tau^3+d\tau^4+H"
 \ee 

 Alors on a :
 
 \begin{theorem}\cite{TP1}.
 Soit
$\psi_0(x,y)=\hbar^{-1/2+\epsilon/2}g(\frac x{\hbar^{1/2-\epsilon}},\frac
y {\sqrt\hbar})$ où $g$ est Schwartz, et $\frac{2\pi}{\sqrt\hbar}$ périodique en $x$. Supposons que, dans (\ref{poi}), 
  $a=\alpha+k2\pi$, $0\leq\alpha<2\pi$. Alors, si $t:=\frac{s2\pi}{b\hbar}$,
 \ben
 e^{i\frac{tH'}\hbar}\psi(x,y)=\frac{1}
 {(\hbar^{1+e\epsilon}sc/4)^{1/4}}g'(|\frac{x}{sc\hbar^{\epsilon/2}}|^{\frac 1 2 },\frac
 y {\sqrt\hbar})+O(\hbar^{\epsilon})
 \end{equation}
 où, 
  si l'on décompose $g(\theta, \eta):=\sum_{j=0}^\infty c_j(\theta)H_j(\eta)$,  les  $H_j$ étant les fonctions d'Hermite,
 
 \[
 g'(|\frac{x}{\hbar^{\epsilon/2}}|^{\frac 1 2 },\eta)=
 \sum_{j\geq0}c_j(|\frac{x+j\alpha}{sc\hbar^{\epsilon/2}}|^{\frac 1 2 })H_j(\eta).
 \]
 \end{theorem}
 En d'autres termes l'état cohérent se relocalise sur la trajectoire en plusieurs sites, générés par les interactions avec les degrés de liberté transverses.

\section{Le cas instable}

Dans cette section nous allons considérer des évolutions à temps suffisamment long pour provoquer des ``reconstructions" comme dans la section précédente, mais cette
fois dans le cas où la dynamique sous-jacente est instable.

\subsection{Le $8$}
Considérons le cas d'une trajectoire homocline : soit $h(x,\xi)$ une fonction $C^\infty$ telle que la courbe $\Omega:=\{(x,\xi)\in\bbR², h(x,\xi)=0\}$ possède un seul point fixe $(x_0,\xi_0)$
, $dh(x_0,\xi_0)=0$, et que celui-ci soit non-dégénéré de type hyperbolique. On peut sans perte de généralité se ramener au cas où $(x_0,\xi_0)$ est à l'origine. De plus il est
bien connu qu'un changement de coordonnées symplectique linéaire (en fait une simple rotation dans l'espace de phases) permet de se ramener au cas où $h(x,\xi)\sim x\xi$ 
près de 0. L'exemple ``typique"est le hamiltonien :
\[
h(x,\xi)=\xi^2+x²(x²-1).
\]
Il est bien connu aussi qu'au niveau quantique, à la fois la translation $(x_0,\xi_0)\to 0$ et la rotation qui donne $h(x,\xi)\to h(x,\xi)\sim x\xi$ près de 0, 
sont réalisés par des opérateur unitaires
explicites (voir la section (\ref{het})). On traitera donc seulement le cas $h(x,\xi)\sim x\xi$ près de 0.

Dénotons encore une fois par $H$ l'opérateur de symbole de Weyl $h(x,\xi)$. Nous allons considérer l'équation de Schrödinger avec pour condition initiale un état cohérent de
``symbole" $a$ et centré à l'origine :
\[
\left\{\begin{array}{c}
i\partial_t\psi=H\psi\\
\psi^{t=0}=\psi^a_{(0,0}:=\psi^a\end{array}\right.
\]
Soit maintenant $(x(s),\xi(s))$ une paramétrisation hamiltonienne de $\Omega$, c'est-à-dire
\[
\left\{\begin{array}{ccc}
\dot x(s)&=&\partial_\xi h(x(s),\xi(s))\\
\dot\xi(s)&=&-\partial_x h(x(s),\xi(s))\end{array}\right.,\ (x(0),\xi(0)):=(x_0,\xi_0).
\]

On a donc $x(s),\xi(s)\to 0$ quand $s\to\pm\infty$. Soit $t_0$
défini par $e^{t_0}=\lim_{s\to+\infty}x(-s)\xi(s)e^{2s}$.
\begin{lemma}
$t_0$ ne dépend pas de $(x_0,\xi_0)$.
\end{lemma}
La preuve est immédiate.

\begin{theorem}\cite{TP2}. 
Soit $0<\gamma<\frac 1 5$ et soit $t_\hbar:=log\frac 1 \hbar-t_0$, alors
\[
e^{-i\frac{t_\hbar H}\hbar}\psi^a=e^{i(S^++\pi/2)/\hbar}\psi^{b_+}+e^{i(S^-+\pi/2)/\hbar}\psi^{b_-}+O(\hbar^{\gamma/2})
\]
où
\[
b_\pm(\eta):=\int_0^{\pm\infty}a(1/\mu)\frac1\mu\rho(\pm\mu\hbar^\gamma)e^{i\eta\mu}d\mu
\]
et $\rho$ est une  fonction de ``cut-off", c'est-à-dire que $$\rho\in C^\infty([0,1]),\ \rho(y)=1,\ \mbox{pour}\ 0\leq y \leq 1, \rho(y)=0,\ \mbox{pour}\ |y|>2$$.
\end{theorem}
Ce théorème indique donc que le paquet d'onde se reforme à l'origine au bout d'un temps logarithmique dans la constante de Planck. De plus il nous indique comment un tel
paquet d'onde re relocalise, en nous donnant la forme de son symbole.

\begin{lemma}\cite{TP2}. 
Définissons sur $\mathcal S(\mathbb R)$ l'opérateur $U$ par, pour $\alpha\in\mathcal S(\mathbb R)$,
\[\begin{array}{crl}
U\alpha(\eta)&:=& e^{i(S^++\pi/2)/\hbar}\int_0^{+\infty}\alpha(1/\mu)\frac1\mu\rho(\mu\hbar^\gamma)e^{i\eta\mu}d\mu\\
&+& 
e^{i(S^-+\pi/2)/\hbar}\int_0^{-\infty}\alpha(1/\mu)\frac1\mu\rho(-\mu\hbar^\gamma)e^{i\eta\mu}d\mu.\end{array}
\]
Alors $U$ est unitaire modulo $\hbar^{\gamma/2}$, 
\[
||U\alpha||_{L^2}=||\alpha||_{L^2}+o(\hbar^{\gamma/2}).
\]
\end{lemma}
Le résultat suivant nous permet d'itérer :

\begin{theorem}\cite{TP2}
Fixons $C>0$. Alors pour tout $n\leq C\frac{log\frac 1 \hbar}{loglog\frac 1 \hbar}$,
\[
e^{-i\frac{nt_\hbar H}\hbar}\psi_a=\psi_{U^na}+O(\hbar^{\gamma/2}(log\frac 1 \hbar)^{n/2}).
\]
Donc la reconstruction semiclassique est valide pour des temps de l'ordre de 
\[
t\sim C\frac{log^2\frac 1 \hbar}{loglog\frac 1 \hbar}.
\]
\end{theorem}
Pour des temps intermédiaires, l´état du système est un distribution lagrangienne associée à $\Omega$.
La  propriété de localisation-delocalisation peut être mieux vue sur les éléments de matrice d'observables : 

\begin{theorem}\cite{TP2}. 
Soit $(p(s),q(s))$ une paramétrisation de $h(p,q)=0$ comme précédemment telle que $q(s)\sim e^s (resp. p(s)\sim e ^{-s})$ quand $s\to -\infty\ (resp. +\infty)$, et soit 
$P$ une observable of
symbole (principal) de Weyl $P(p,q)$. Alors, uniformément pour  $t\leq (1+\epsilon)log \frac 1 \hbar,\ 0\leq\epsilon<1/2$,
\[\begin{array}{rcl}
<\psi^t,P\psi^t>&=&<\psi_a,e^{i\frac{t H}\hbar}Pe^{-i\frac{t H}\hbar}\psi_a>\\
&=&\hbar^{-1/2}\int_{-\infty}^{+\infty}|a(e^{s-t}/\hbar)|^2P(p(s),q(s))e^{s-t}ds +O(\hbar^{11/2})\end{array}.
\]
En particulier, pour $t$ borné : 
\[
<\psi^t,P\psi^t>=P(0,0)+O(\hbar^{1/2}),
\]
pour $t\sim log\frac 1 \hbar$ :
\[
<\psi^t,P\psi^t>= P(0,0) +O(\hbar^{\gamma/2}),
\]
et pour $t= \frac 1 2 log \frac 1 \hbar-t'$ :
\[
<\psi^t,P\psi^t>=\int_{h(p,q)=0}P(p,q)d\mu^{t'}+O(\hbar^{1/2})
\]
où la mesure  $d\mu^{t'}$ est donnée par la densité $|a(e^{s-t'})|^2e^{s-t'}$, c'est-à-dire $d\mu^{t'}=|a(e^{s-t'}|^2e^{s-t'}ds$.
De plus les même résultats restent valides après $n$ itérations ($n\leq C\frac{log\frac 1 \hbar}{loglog\frac 1 \hbar}$) : on a, pour $t\sim n log\frac 1 \hbar$:
\[
<\psi^t,P\psi^t>= P(0,0) +O(\hbar^{\gamma/2}(log\frac 1 \hbar)^{n/2}),
\]
et pour $t= \frac n 2 log \frac 1 \hbar-t'$:
\[
<\psi^t,P\psi^t>=\int_{h(p,q)=0}P(p,q)d\mu^{t'}+O(\hbar^{\gamma/2}(log\frac 1 \hbar)^{n/2}),
\]
où la mesure  $d\mu^{t'}$ est donnée par la densité $|U^na(e^{s-t'})|^2e^{s-t'}$
\end{theorem}

\subsection{Interprétation}
Nous voyons donc apparaître, à la limite semiclassique, des oscillations en temps entre états localisés pour des temps multiples de $\log{\frac ` \hbar}$ et des états
nonlocalisés pour des temps intermédiaires. De plus l'application qui fait passer du symbole de l'état cohérent original à celui reconstruit
après une itération joue le rôle de la quantification (métaplectique) du flot linéarisé.

\subsection{Le cas hétérocline quelconque}\label{het}

Dans cette  section nous indiquons comment traiter le cas  général d'une jonction hétérocline entre deux points fixes hyperboliques. Considérons un hamiltonien $h(p,q)$ tel que la surface
d'énergie  $\{h=0\}$ contienne deux points fixes hyperboliques reliés par une courbe régulière. On peut considérer sans perte de généralité 
que $h$ est tel que:
\[
h(p,q)=pq + O(|(p,q)|^3)
\]
et
\[
dh(p_0,q_0)=0
\]
avec $dh\neq 0$ sur la courbe $\Lambda$ entre l'origine et $(p_0,q_0)$. Supposons de plus que le  hessien de $h$ à $(p_0,q_0)$, $Hess(h)_{p_0,q_0}$, est de type hyperbolique, 
c'est-à-dire qu'il existe une matrice  $2\times 2$ symplectique
$M$ telle que
\[
M^THess(h)_{p_0,q_0}M=\left(\begin{array}{cc} \mu_0&0\\0&\frac 1 {\mu_0}\end{array}\right).
\]
Il est bien connu que $M$ est une rotation dans l'espace de phases. Soit $\theta_0$ l'angle de cette rotation. A $M$ on peut associer l'opérateur (métaplectique)  $\tilde{M}$ 
unitaire sur $L^2(\bbR, d\eta)$ défini par :
\[
\tilde M=e^{i\frac{\theta_0}2(-\frac{d^2}{d\eta^2}+\eta^2)}
\]
(notons que le noyau intégral de $\tilde M$ est calculable explicitement).
Considérons une condition initiale qui est un état cohérent à $(0,0),\ \psi^{a}$, nous allons calculer l'évolution de la partie $\psi^{a+}$ qui correspond à la partie de la condition
initiale ``vers $(p_0,q_0)$", c'est-à-dire moralement :
\[
\psi^{a+}(x)=\psi^{a}(x),\ x>0:\ \ \ \psi^{a+}(x)=0,\ x<0.
\]
dans le cas où la partie $\Lambda$ de la surface d'énergie reliant l'origine à $(x_0,\xi_0)$ est tangente à l'origine à $\bbR^+$. $\psi^{a+}$ est trop singulière
pour que l'on puisse calculer semiclassiquement $ e^{-i\frac{t_\hbar H}\hbar}\psi_{a+}$, il faut régulariser un peu : écrivons, pour $\epsilon>0$,
\[
a=a_\epsilon^++a_\epsilon^-
\]
avec  $||a_\epsilon^\pm||_{L^2(\bbR^\mp)}=O(\epsilon)$ et $a_\epsilon^\pm$ dans la classe de Schwartz.

\begin{theorem}\cite{TP2}. 
Soit $H$ la quantification de Weyl de $h$. Alors il existe $t_0$ tel que, si
$t_\hbar:=(\frac 1 2 +\frac 1 {\mu_0})\log{\frac 1 \hbar}-t_0$, on a

\[
e^{-i\frac{t_\hbar H}\hbar}\psi^{a_\epsilon^+}=e^{i(S^++\sigma\pi/2)/\hbar}\psi^{b_+}+O(\hbar^{\gamma/2})+O(\epsilon),
\]
où $\sigma$ est l'indice de  Maslov de la courbe $\Lambda$, $S^+:=\int_\Lambda pdq$ et
\[
b_+=\tilde{M}a(\frac 1 .)\frac 1 .\rho(.\hbar^\gamma).
\]
\end{theorem}

Le théorème précédent permet clairement de considérer une jonction hétérocline quelconque. Regardons quelques exemples.

\subsection{Le pendule}
Considérons le hamiltonien suivant
\[
h^{PENDULE}(p,q)=\frac {p^2} 2+cosq-1
\]
Nous le transformerons tout d'abord, comme précédemment, en :
\[
h(p,q)=\frac{(p+q)^2}4+cos\frac{(p-q)}{\sqrt2}-1
\]
Donc $h(p,q)\sim pq$ près de l'origine.

Soit $\Lambda:=\{m(1,1)\in\mathbb Z^2, m\in\bbZ\}$ et $\mathcal O^n$ l'ensemble de tous les chemins connexes dans $\Lambda$ partant de  l'origine et de longueur $n$.
Un élément de 
$\mathcal O$ est une suite $(\lambda_i)_{i=0\dots n}/\lambda_0=(0,0),\  |\lambda_{i+1}-\lambda_i|=\sqrt 2$. Soit encore une fois $\psi^a_{(q,p)}$ défini par (\ref{ec}).
\begin{theorem}\cite{TP2}.
\[
e^{-i\frac{nt_\hbar H}\hbar}\psi^a_{(0,0)}=\sum_{\lambda\in\mathcal O^n}e^{iS_\lambda/\hbar}\psi^{a_\lambda}_{\lambda_n}+O(\hbar^{\gamma/2}(log\frac 1 \hbar)^{n/2}),
\]
où $S_\lambda=\int_{\tilde{\lambda}}pdq$ , avec $\tilde{\lambda}$ la courbe dans $\{h(p,q)=0\}$ joignant les points de $\lambda$, et 
\[
a_\lambda=\Pi_{i=1}^nT_{\lambda_i}:=T_\lambda a
\]
avec 
\[
T_{\lambda_i}b(\eta)=\int_0^{+\infty}e^{-i\mu\eta}a(\frac 1 \mu)\frac 1 \mu\rho(\mu\hbar^\gamma)d\mu
\]
if $\lambda_{i+1}-\lambda_i>0$ et 
\[
T_{\lambda_i}b(\eta)=\int_0^{-\infty}e^{+i\mu\eta}a(\frac 1 \mu)\frac 1 \mu\rho(-\mu\hbar^\gamma)d\mu
\]
si $\lambda_{i+1}-\lambda_i<0$.
\end{theorem}

\subsection{Le cas Harper}

Le modèle de Harper est donné par le hamiltonien
\[ 
h^{HARPER}(p,q):=cos(p)-cos(q)
\]
Par un simple changement de variable il peut être ramené à :
\[
h(p,q):=\pi^2(cos((p+q)/2\pi)-cos((p-q)/2\pi))
\]
avec $h(p,q)\sim pq$ près de zéro.

\begin{theorem}\cite{TP2}. 
Soit $\OE^n$ (pour \OE dipe) l'ensemble des chemins $\Gamma$ sur $\mathbb Z^2$ partant de  $(0,0)$ et ne contenant aucune segment de longueur plus grande que $1$. 
Soit $\Gamma(n)$ l'extrémité de $\Gamma$ et $\Gamma_i$ un vertex de $\Gamma$. Soit $t_\hbar=log\frac1 h\hbar$. Alors
\[
e^{-i\frac{nt_\hbar H}\hbar}\psi^a_{(0,0)}=\sum_{\Gamma\in\OE^n}e^{iS_\Gamma/\hbar}\psi^{a_\Gamma}_{\Gamma(n)}+O(\hbar^{\gamma/2}(log\frac 1 \hbar)^{n/2}),
\]
où $S_\Gamma=\frac 1 2 \int_{\tilde{\Gamma}} pdq-qdp$, où $\tilde{\Gamma}$ est chemin dans $\mathbb R^2$ consistant en les segments joignant les points de $\Gamma$ et
\[
a^\Gamma=\Pi_{i=1}^n V^{\Gamma_i} a:=V_\Gamma a
\]
où :
\begin{itemize}
\item 
$
V^{\Gamma_i}a(\eta)=a(1/\eta)\frac{1}\eta\rho(\eta\hbar^\gamma)
$
si $(\Gamma_{i-1},\Gamma_i)$  est horizontal  orienté à droite,
$
V^{\Gamma_i}a(\eta)=-a(-1/\eta)\frac{1}\eta\rho(-\eta\hbar^\gamma)
$
si $(\Gamma_{i-1},\Gamma_i)$ est horizontal  orienté à gauche,
\item
$
V^{\Gamma_i}a(\eta)=\int_0^\infty e^{-i\eta\mu}a(1/\mu)\rho(\mu\hbar^\gamma)\frac{d\mu}\mu
$
si $(\Gamma_{i-1},\Gamma_i)$  est vertical  orienté vers le haut, 
\item
$
V^{\Gamma_i}a(\eta)=\int_0^{-\infty} e^{i\eta\mu}a(1/\mu)\rho(-\mu\hbar^\gamma)\frac{d\mu}\mu
$
si $(\Gamma_{i-1},\Gamma_i)$  est horizontal  orienté à droite.
\end{itemize} 

\end{theorem}

Une autre manière d´énoncer le même  phénomène est le résultat suivant, inspiré par les  ``intégrales de chemins" de Feynman.
\begin{corollary}\cite{TP2}. 
Soit $n\leq C\frac{log\frac 1 \hbar}{loglog\frac 1 \hbar}$. Considérons éléments de matrice
\[
<\psi^a_{(0,0)},e^{-i\frac{nt_\hbar H}\hbar}\psi^b_{(p,q)}>.
\]
Ceux-ci aurons une limite semiclassique non-nulle seulement si $(p,q)=(i,j)\in\bbZ^2$ et
\[
<\psi^a_{(0,0)},e^{-i\frac{nt_\hbar H}\hbar}\psi^b_{(i,j)}>=\sum_{\Gamma\in\OE,\ \Gamma(n)=(i,j)}e^{iS_\Gamma/\hbar}<a,V_\Gamma b>+
O(\hbar^{\gamma/2}(log\frac 1 \hbar)^{n/2}).
\]
la somme devant être comprise comme étant nulle s'il n'y pas de chemin satisfaisant
 $\Gamma(n)=(i,j)$.
\end{corollary}

\section{Généralisations}

Les exemples traités plus haut se généralisent à des espaces de Hilbert qui ne sont pas du type $L^2$. C'est le cas, par exemple, des spins. L'espace de Hilbert (de la
partie ``spin") d'une particule de spin $\frac 1 2 $ est $\bbC²$. C'est un des espaces de représentation du groupe des rotations de $\bbR³$.

Un autre généralisation qui nous semble plausible est le cas de systèmes intégrables. 
Dans le cas d'un système intégrable quantique où l'on peut écrire le hamiltonien  comme fonction d'opérateurs qui commutent entre eux et de spectres linéaires, alors le
résultat doit être immédiat, et la généralisation des résultats de la section \ref{sta} facile.

L'autre cas est celui de tores singuliers, présentant des points hyperboliques. Un exemple est traité dans \cite{DP}.

\section{Indéterminisme quantique et sensibilité aux conditions initiales}
Comme nous l'avons mentionné dans l'introduction, le processus de mesure de la \mq est aussi une façon de ``retrouver" le monde classique. Rappelons brièvement que la mesure
quantique est associée à une observable quantique $\mathcal O$ (opérateur auto-adjoint) et que le résultat d'une mesure effectuée peut être un quelconque point $\lambda$ 
du spectre de
$\mathcal O$, la probabilité d'obtenir $\lambda$ étant  égale à $|<\psi|\phi_\lambda>|^2$, où $\psi$ est l'état du système au moment où la mesure est
effectuée, et $\phi_\lambda$ le vecteur propre (éventuellement généralisé si le spectre n'est pas ponctuel) associé à $\lambda$, supposée non-dégénérée. Dans le cas où $\mathcal O=Q$,
où $Q$ est l'opérateur de multiplication par la variable $x$, on obtient l'interprétation probabiliste de la fonction d'onde : la probabilité qu'a une particule quantique d'être dans la
position $x$ est égale à $|\psi(x)|^2$.

Cet aspect non-déterministe est à opposer au   caractère fondamentalement déterministe de la mécanique classique, intrinsèquement associée à un flot sur un espace
géométrique absolu. Cependant la remarque, par Poincaré, de l'importance de la sensibilité aux conditions initiales rend souvent inapplicable ce déterminisme. On sait par exemple
que pour savoir si, dans un million d'années, la terre sera encore dans le système solaire, il faudrait connaître avec un précision de quelques centaines de mètres la
position de toutes les planètes du système solaire : une mesure classique tout simplement invraisemblable à obtenir. Dans l'asymptotisme du temps infini cette remarque prend
un sens tout à fait particulier : bien que le flot classique associé à un hamiltonien (régulier) est défini  univoquement  pour tout temps, le flot à temps
infini n'existe pas (en tant que flot). Il intervient néanmoins partout dans la théorie des systèmes dynamiques (ergodicité, sensibilité aux conditions initiales,
propriété de mélange etc), mais de façon indirecte. Faible pour l'ergodicité :
\[
\lim_{T\to\infty}\frac 1 T\int_{-T/2}^{+T/2}f(\phi^t(x))dt=\mu(x)
\]
et implicite pour la sensibilité aux conditions initiales (le flot étant continu pour chaque $t<\infty$) :
\[
\exists I\in\bbR^+ , \forall\epsilon>0,\  \exists t=t(I,\epsilon)\ \mbox{tel que}\ \exists x, |x-y|\leq\epsilon,\ |\Phi^t(x)-\Phi^t(y)|\geq I.
\]

Il y a pourtant une situation où le flot à temps infini peut être partiellement défini : il s'agît de la situation 
d'un point fixe hyperbolique (et plus généralement d'une trajectoire instable dans un système chaotique) : la variété instable d'un point fixe $z_0$, $\Lambda^-_{z_0}$ est en effet définie 
comme l'ensemble des points $z$ tels que $\phi^{+\infty}(z)=z_0$ :
\[
\Lambda^-_{z_0}=\{z/ \phi^{+\infty}(z)=z_0\}.
\]
 On pourrait donc être tenté de ``définir'' le flot inverse de la manière suivante :
\[
\phi^{-\infty}(z_0)=z,\ \forall z\in\Lambda_{z_0},
\]
c'est-à-dire de manière non-déterministe.

Cette définition n'a bien sûr pas de sens à l'intérieur du paradigme classique, mais se retrouve prendre une signification dans le cadre semiclassique : la délocalisation de la
fonction d'onde pour une échelle de temps logarithmique, couplée à l'interprétation probabiliste de celle-ci, donne précisément un sens non-déterministe 
au flot limite à temps infini (voir \cite{TP4} pour une discussion plus détaillée).

 \section{Conclusion : périodes semiclassiques} 

Nous avons vu comment la mémoire des phases permet la reconstruction de paquets d'ondes et induit donc un flot \textit{périodique} dans le temps infini semiclassique.

Nous avons aussi vu que les échelles de temps, pour cette reconstruction, sont différentes pour les cas stable/instable.
\[
\left\{\begin{array}{c}
\hbar^{-1} \ \mbox{pour le cas table}\\
\log{\frac 1 \hbar}\ \mbox{pour le cas instable}\end{array}\right.
\]
Ces échelles de temps, lorsque $\hbar\to 0$ ``disparaissent" dans le point à l'infini du temps classique. Les exemples que nous avons traités font donc appréhender une 
approche du
temps infini en mécanique classique dans laquelle le ``point" à l'infini est décomposé en différentes échelles, suggérant un nouveau système dynamique (vraiment semiclassique et
non-classique). 

Nous avons enfin vu la nécessité pour les états cohérents de n'être plus gaussien, les reconstructions ne préservant pas cette propriété. De plus les symboles des
états reconstruits sont, dans tous les cas, singuliers lorsque considérés ponctuellement.

\end{document}